\documentclass[reqno,11pt]{amsart}

\usepackage{amsmath,amsfonts,amssymb,amsthm,epsfig,amscd,}
\usepackage{color,comment}
\usepackage[]{fontenc}
\usepackage[latin1]{inputenc}
\usepackage{enumerate}
\usepackage[cmtip,all]{xy}
\usepackage{tabularx,multirow}
 \allowdisplaybreaks \numberwithin{equation}{section}

\newtheorem{theorem}{Theorem}[section]
\newtheorem{proposition}[theorem]{Proposition}
\newtheorem{corollary}[theorem]{Corollary}
\newtheorem{lemma}[theorem]{Lemma}

\newtheorem{namedtheorem}{Theorem}

\theoremstyle{definition}

\newtheorem{question}[theorem]{Question}

\theoremstyle{remark}
\newtheorem{remark}[theorem]{Remark}

\def\bP{{\mathbb P}}

\begin{document}


\author{Nicola Picoco}
\title{Geometry of points satisfying Cayley-Bacharach conditions and applications}
\address{Dipartimento di Matematica, Universit\`a degli Studi di Bari, Via Edoardo Orabona 4, 70125 Bari -- Italy}
\email{nicola.picoco@uniba.it}
\thanks{}
\date{}
\maketitle

\begin{abstract}
In this paper, we study the geometry of points in complex projective space that satisfy the Cayley-Bacharach condition with respect to the complete linear system of hypersurfaces of given degree. In particular, we improve a result by Lopez and Pirola and we show that, if $k\geq 1$ and $\Gamma =\{P_1,\dots,P_d\}\subset \mathbb{P}^n$ is a set of distinct points satisfying the Cayley-Bacharach condition with respect to $|\mathcal{O}_{\mathbb{P}^n}(k)|$, with $d\leq h(k-h+3)-1$ and $3\leq h\leq 5$, then $\Gamma$ lies on a curve of degree $h-1$. Then we apply this result to the study of linear series on curves on smooth surfaces in $\mathbb{P}^3$. Moreover, we discuss correspondences with null trace on smooth hypersurfaces of $\mathbb{P}^n$ and on codimension $2$ complete intersections.
\end{abstract}
\medskip

\section{Introduction}

Let $\mathcal{D}$ be a complete linear system on a smooth projective variety $X$. We say that a set of points $\Gamma=\{P_1,\dots,P_d\}$ satisfies the \emph{Cayley-Bacharach condition with respect to $\mathcal{D}$} if for every $i=1,\dots,d$ and for any effective divisor $D\in\mathcal{D}$ passing through $P_1,\dots,\widehat{P_i},\dots,P_d$, we have $P_i\in D$ as well. In particular, if a finite set of points $\Gamma\subset\mathbb{P}^n$  satisfies the Cayley-Bacharach condition with respect to the complete liner system $|\mathcal{O}_{\mathbb{P}^n}(k)|$ of hypersurfaces of degree $k$, we will say  that $\Gamma$ is $CB(k)$.

This definition of the Cayley-Bacharach condition is the modern version of a very classical property that can be found since some milestone results of ancient geometry (see \cite{EGH} for a detailed historical background). A turning point towards a modern point of view was the celebrated theorem of Cayley and Bacharach; it asserts that if $\Gamma=\{P_1,\dots,P_{de}\}\subset\mathbb{P}^2$ is the collection of intersection points between two plane curves of degree $d$ and $e$ respectively, then the set $\Gamma$ is $CB(d+e-3)$.

Although the Cayley-Bacharach condition is linked to classical issues, it represents a very fruitful field with many open questions and various applications. For instance, in the last decade there has been a new and growing interest for the Cayley-Bacharach property due to its applications to the study of \emph{measures of irrationality} for projective varieties (cf. \cite{LP}, \cite{B}, \cite{BCD}, \cite{BDELU} and \cite{GK}). Furthermore,  this property has been recently studied from an algebraic viewpoint in terms of $0$-dimensional affine algebras over arbitrary base fields (cf. [KLR]).\\

In this paper we are concerned with the geometry of points $P_1,\dots,P_d\in\mathbb{P}^n$ satisfying the Cayley-Bacharach condition. An important result in this direction is a theorem by Lopez and Pirola that states that if $d \leq h(k -h +3)- 1$ for some $h\geq 2$, then $\Gamma$ lies on a curve of degree $h - 1$. Furthermore, they prove the same result also in $\mathbb{P}^3$ but only for $2 \leq h \leq 4$ (cf. \cite[Lemma 2.5]{LP}).

It is therefore natural to investigate whether it remains true in any $\mathbb{P}^n$ and for all $h\geq 2$. More precisely, we can formulate the following question.

\begin{question}\label{question} Let $\Gamma$ be a set of $CB(k)$ distinct points in $\mathbb{P}^n$ such that  $|\Gamma|\leq h(k -h +3)- 1$. Does $\Gamma$ lie on a curve of degree $h-1$ for any $h\geq 2$?\end{question}

The case $h=2$ is completely understood. Namely, in \cite[Lemma 2.4]{BCD} the authors prove that if $\Gamma\subset\mathbb{P}^n$ is a set of distinct points $CB(k)$ and $|\Gamma|\leq 2k+1$, then $\Gamma$ lies on a line.

We deal with the cases $h=3,4$ and $5$ proving the following theorem which summarizes Theorems \ref{Th3},  \ref{Th4} and  \ref{Th5}.

\begin{namedtheorem}\label{ThA}
Let $\Gamma =\{P_1,\dots,P_d\}\subset \mathbb{P}^n$ be a set of distinct points satisfying the Cayley-Bacharach condition with respect to $|\mathcal{O}_{\mathbb{P}^n}(k)|$, with $k\geq 1$. For any $3\leq h\leq 5$, if
\begin{equation}\label{boundLP}
d\leq h(k-h+3)-1
\end{equation}
then $\Gamma$ lies on a curve of degree $h-1$.
\end{namedtheorem}

We recall that cases $h=3,4,5$ in $\bP^2$ and cases $h=3,4$ in $\bP^3$ are covered by \cite[Lemma 2.5]{LP}, so Theorem \ref{ThA} is an extension of this result.
As a byproduct of Theorem \ref{ThA} we get some improvements of results in \cite{SU} and \cite{LU}. Let $\Gamma$ be a set of points $CB(k)$. In  \cite[Theorem 1.9]{SU} the authors prove that $\Gamma$ lies on a curve of degree at most $2$ provided that $|\Gamma|\leq \frac{5}{2}k+1$. We note that this bound is smaller than $3k-1$ (i.e. the bound (\ref{boundLP}) with $h=3$) as soon as $k\geq 4$, so Theorem \ref{ThA} does provide an improvement. Moreover, by Theorem \ref{Th5}, we get a partial improvement of \cite[Theorem 1.3 (iii)]{LU} which ensures that if $\Gamma$ is $CB(k)$ and $|\Gamma|$ is bounded by a value depending on $k$ and on an integer $d$, then $\Gamma$ lies on a union of positive-dimensional linear subspace of $\mathbb{P}^n$ whose dimensions sum to $d$ (cf.  Remark \ref{improveConjLU}).

Concerning the proof of Theorem \ref{ThA}, the cases $h=3,4$ are achieved by induction on the dimension of the ambient space $\mathbb{P}^n$, using the fact that the assertion for $n=2,3$ is covered by \cite[Lemma 2.5]{LP}. On the other hand, the case $h=4$ is more complicated, because the induction argument requires to prove separately the cases $n=3$ and $n=4$. To this aim we extend to this setting the argument of \cite{LP}. In particular, we first prove that $\Gamma$ lies on a reduced curve $C$ of degree at most $9$. Then we distinguish several cases (depending on the irreducible components of $C$) and combining \cite[Theorem 1.5]{G}, the main result of \cite{GLP} and properties of points $CB(k)$ (see e.g. \cite{LU}) we show that the sum of the degrees of the irreducible components $C_i$ of $C$ such that $C_i\cap \Gamma\neq \emptyset$ is at most $4$, as wanted.\\

In the last section we present various applications of Theorem \ref{ThA}.  One of those concerns the study of linear series on curves lying on smooth surfaces in $\mathbb{P}^3$. Following the very same argument in \cite{LP}, we prove the following slight improvement of \cite[Theorem 1.5]{LP}.

\begin{namedtheorem}\label{THlinearseries}
Let $S\subset\mathbb{P}^3$ be a smooth surface of degree $d\geq 5$, $C$ an integral curve on $S$ such that $|\mathcal{O}_C\otimes \mathcal{O}_S(C)|$ is base point free and $g^r_n$ a base point free special linear series on $C$ that is not composed with an involution if $r\geq 2$. If $n\leq 5d-31$ there exists an integer $h$, with $1\leq h\leq 4$, such that 
\begin{equation}\label{hgrnOnSurface}
h(d-h-1)\leq n\leq \min\{hd,(h+1)(d-h-2)-1\}
\end{equation}
and the general divisor of the $g^r_n$ lies on a curve of degree $h$.
\end{namedtheorem}

Our contribution is the improvement of the upper bound on $n$ (that previously was $n\leq 4d-21$) allowing the case $h=4$. The crucial point in this theorem is that the support of a general divisor on suitable linear series on a curve that moves on a smooth surface $S$ in $\mathbb{P}^3$ satisfies the Cayley-Bacharach condition with respect to $|K_S|$ (cf. \cite[Lemma 3.1]{LP}). 

Next we present an improvement of \cite[Theorem 1.3]{LP} concerning the so-called \emph{correspondences with null trace} on surfaces in $\mathbb{P}^3$ (see Subsection \ref{subcor}) and we extend it to any hypersurfaces in $\bP^n$. It is worth noticing that correspondences with null trace on algebraic varieties have been recently discussed in \cite{LM}, together with some interesting results on their degree in the case of very general hypersurfaces. We prove the following.

\begin{namedtheorem}\label{THcorrispnull}
Let $n\geq 3$ and let $X\subset\mathbb{P}^n$ be a smooth hypersurface of degree $d\geq n+2$. Let $\Gamma$ a correspondence of degree $m$ with null trace on $Y\times X$. Then if $m\leq 5(d-n)-16$, the only possible values of $m$ are $d-n+1\leq m\leq d$, $2(d-n)\leq m\leq 2d$, $3(d-n-1)\leq m\leq 3d$ and $4(d-n-2)\leq m\leq 4d$.
\end{namedtheorem}

Actually we prove a slightly finer statement, that is, if $m\leq h(d-n-h+2)-1$ for $2\leq h\leq 5$, then the only possible values of $m$ are $(s-1)(d-n-s+3)\leq m\leq (s-1)d$ for $2\leq s\leq h$ (cf. proof of Proposition \ref{strongTHcorrispnull}). Moreover, we point out that the bound $m\geq d-n+1$ was already obtained in \cite[Theorem 1.2]{BCD}.

Finally, we give a partial extension of \cite[Theorem 1.4]{LU} about the geometry of fibers of low degree maps to projective space from certain codimension two complete intersections (cf. Corollary \ref{corrispAndPlanConf}) by studying this problem in terms of correspondences with null trace on these varieties (cf. Proposition \ref{planeconfOfCorresp}).\\

The paper is organized as follows. In Section 2 we recall some basic properties of sets of points in $\mathbb{P}^n$ that satisfies Cayley-Bacharach condition with respect to the complete liner system $|\mathcal{O}_{\mathbb{P}^n}(k)|$ and we prove some technical results. Section 3 is devoted to prove case $h=3$ and case $h=4$ of Theorem \ref{ThA}. In Section 4 we deal with the proof of case $h=5$ of Theorem \ref{ThA}.  Finally, in Section 5, we are concerned with applications of Theorem \ref{ThA} to linear series on curves, correspondences with null trace and plane configurations of points in $\bP^n$.\\
\medskip

\section{Properties of points satisfying Cayley-Bacharach condition}

We collect some useful properties of sets that satisfies the Cayley-Bacharach condition with respect to the complete liner system $|\mathcal{O}_{\mathbb{P}^n}(k)|$ of hypersurfaces of degree $k$.

\begin{lemma}\label{minumCB}
Let $\Gamma = \{p_1,\dots,p_m\}\subset\mathbb{P}^n$ be a set of distinct points that is $CB(k)$. If $\Gamma$ is non-empty, then $m\geq k+2$.
\end{lemma}

\begin{proof} See \cite[Lemma 2.4]{BCD}. \end{proof}

\begin{proposition}\label{pointsoutofCB}
Let $\Gamma = \{p_1,\dots,p_m\}\subset\mathbb{P}^n$ be a set of distinct points $CB(k)$ and let $\mathcal{P}=L_1\cup\dots\cup L_r$ be a union of positive-dimensional linear spaces. Then $\Gamma\backslash \mathcal{P}$, if non-empty, is $CB(k-r)$.\\
In particular, the complement of $\Gamma$ by a single linear space is $CB(k-1)$.
\end{proposition}

\begin{proof} See \cite[Proposition 2.5]{LU}. \end{proof}

\begin{remark}
In the introduction we mentioned the notion of plane configuration. We will return to this topic in Subsection \ref{subPlaneConf}. Here we anticipate that a plane configuration of length $r$ is a union $\mathcal{P}=L_1\cup\dots\cup L_r$ of positive-dimensional linear spaces. Therefore, the previous proposition can be restate using this terminology.
\end{remark}

\begin{proposition}\label{allineati}
Let $\Gamma = \{p_1,\dots,p_m\}\subset\mathbb{P}^n$ be a set of distinct points $CB(k)$. If $m\leq 2k+1$, then $\Gamma$ lies on a line.
\end{proposition}

\begin{proof} See \cite[Lemma 2.4]{BCD}. \end{proof}

In \cite{SU}, Stapelton and Ullery proved that following result.

\begin{proposition}\label{SUresult}
Let $\Gamma = \{p_1,\dots,p_m\}\subset\mathbb{P}^n$ be a set of distinct points $CB(k)$, with $k\geq 1$. If $m\leq \frac{5}{2}k+1$, then $\Gamma$ lies on a reduced curve of degree $2$.
\end{proposition}

\begin{proof} See \cite[Theorem 1.9]{SU}. \end{proof}

The following proposition, that is crucial in Section 4, gives conditions under which a set of points lying on an integral curve can not be $CB(k)$ for some $k\geq 1$.

\begin{proposition}\label{LGreco}
Let $\Gamma=\{P_1,\dots,P_m\}\subset \mathbb{P}^n$ be a set of points. Let $C\subset \mathbb{P}^n$ be an integral curve passing through all the points of the set $\Gamma$ and with the dualizing sheaf $\omega_{C}$ invertible. Let $Z=P_1+\dots +P_m$ and $E=P_1+\dots +P_{m-1}$.  If $k\geq 1$ is an integer such that $H^1(C, \mathcal{O}_{C}(k)(-Z))=H^1(C, \mathcal{O}_{C}(k)(-E))=0$ and $H^0(\mathbb{P}^n,\mathcal{O}_{\mathbb{P}^n}(k))\to H^0(C,\mathcal{O}_C(k))$ is surjective, then $\Gamma$ can not be $CB(k)$.
\end{proposition}

\begin{proof}
We denote $\mathbb{P}^n$ by $\mathbb{P}$. An integral curve $C$ with the dualizing sheaf invertible is Gorenstein  (cf. \cite{H1}, V ,  \S 9 ). Moreover, for any $0$-dimensional subscheme  $Z$  of length $d$ on the integral Gorenstein curve $C$, Riemann-Roch Theorem (cf. \cite{H86} Theorem 1.3) ensures that 
$$h^0(\mathcal{L}(Z))-h^1(\mathcal{L}(Z))=d+1-p_a$$
where $p_a$ is the arithmetic genus of $C$.

Thus, for $kH-Z$ and $kH-E$, we get
$$h^0(\mathcal{O}_{C}(k)(-Z))-h^1(\mathcal{O}_{C}(k)(-Z))=\text{deg}(kH-Z) +1 -p_a,$$
$$h^0(\mathcal{O}_{C}(k)(-E))-h^1(\mathcal{O}_{C}(k)(-E))=\text{deg}(kH-E) +1 -p_a.$$
By assumption, $h^1(\mathcal{O}_{C}(k)(-Z))=h^1(\mathcal{O}_{C}(k)(-E))=0$ and deg$(kH-E)>$deg$(kH-Z)$. Thus $h^0(\mathcal{O}_{C}(k)(-Z))\neq h^0(\mathcal{O}_{C}(k)(-E))$. 

We claim that $H^1(\mathbb{P},\mathcal{O}_{\mathbb{P}}(k)(-C))=0$. In order to see this, let us consider the exact sequence
$$0\to \mathcal{O}_{\mathbb{P}}(-C) \to \mathcal{O}_\mathbb{P} \to \mathcal{O}_C \to 0.$$
Twisting by $k$ and passing to cohomology, we get
$$0\to H^0(\mathbb{P},\mathcal{O}_{\mathbb{P}}(k)(-C)) \xrightarrow{\alpha} H^0(\mathbb{P},\mathcal{O}_\mathbb{P}(k)) \xrightarrow{\beta} H^0(C,\mathcal{O}_C(k)) \xrightarrow{\gamma}  H^1(\mathbb{P},\mathcal{O}_{\mathbb{P}}(k)(-C)) \to 0,$$
where the last element of the sequence is $0$ since $H^1(\mathbb{P},\mathcal{O}_\mathbb{P}(k))=0$ for any $k\geq 1$.
Thus $\gamma$ is surjective and Ker$\gamma=$Im$\beta$. Moreover, $\beta$ is surjective by assumption and hence Ker $\gamma=H^0(C,\mathcal{O}_C(k))$. Then the claim follows.

Let us now consider the following exact sequences
$$ 0\to \mathcal{O}_{\mathbb{P}}(-C) \to \mathcal{O}_{\mathbb{P}}(-Z) \to \mathcal{O}_{C}(-Z) \to 0,$$
$$ 0\to \mathcal{O}_{\mathbb{P}}(-C) \to \mathcal{O}_{\mathbb{P}}(-E) \to \mathcal{O}_{C}(-E) \to 0. $$
From these, we obtain
$$0\to H^0(\mathbb{P},\mathcal{O}_{\mathbb{P}}(k)(-C)) \xrightarrow{\alpha_Z} H^0(\mathbb{P},\mathcal{O}_{\mathbb{P}}(k)(-Z)) \xrightarrow{\beta_Z} H^0(C,\mathcal{O}_{C}(k)(-Z))  \to 0,$$
$$0\to H^0(\mathbb{P},\mathcal{O}_{\mathbb{P}}(k)(-C)) \xrightarrow{\alpha_E} H^0(\mathbb{P},\mathcal{O}_{\mathbb{P}}(k)(-E)) \xrightarrow{\beta_E} H^0(C,\mathcal{O}_{C}(k)(-E))  \to 0,$$
where the last element of both sequences is $H^1(\mathbb{P},\mathcal{O}_{\mathbb{P}}(k)(-C))=0$.\\

It follows that
$$H^0(C,\mathcal{O}_{C}(k)(-Z))\cong \frac{H^0(\mathbb{P},\mathcal{O}_{\mathbb{P}}(k)(-Z))}{H^0(\mathbb{P},\mathcal{O}_{\mathbb{P}}(k)(-C))}$$
and
$$H^0(C,\mathcal{O}_{C}(k)(-E))\cong \frac{H^0(\mathbb{P},\mathcal{O}_{\mathbb{P}}(k)(-E))}{H^0(\mathbb{P},\mathcal{O}_{\mathbb{P}}(k)(-C))}.$$
Thus $H^0(\mathbb{P},\mathcal{O}_{\mathbb{P}}(k)(-Z))\not\cong H^0(\mathbb{P},\mathcal{O}_{\mathbb{P}}(k)(-E))$, since we saw that $h^0(\mathcal{O}_{C}(k)(-Z))\neq h^0(\mathcal{O}_{C}(k)(-E))$. Therefore the set $\Gamma=\{P_1,\dots,P_m\}\subset \mathbb{P}^n$ does not satisfies the Cayley-Bacharach condition with respect to the complete linear system $|\mathcal{O}_{\mathbb{P}^n}(k)|$.
\end{proof}

The previous proposition has an abstract formulation that does not lend itself to a practical application. The next corollary achieves exactly this goal.

\begin{corollary}\label{CGreco}
Let $\Gamma=\{P_1,\dots,P_m\}\subset \mathbb{P}^n$ be a set of points and let $C\subset \mathbb{P}^n$ be a non degenerate integral complete intersection curve passing through all the points of the set $\Gamma$. If
$$i) \;\;\;   k> \frac{m+2p_a-2}{d} \hspace{1cm} \text{and} \hspace{1cm} ii) \;\;\;  k\geq d+1-n $$
where $d$ is the degree of the curve $C$ and $p_a$ its arithmetic genus, then $\Gamma$ can not be $CB(k)$. 
\end{corollary}

\begin{proof}  If $C$ is a complete intersection, then $\omega_{C}$ is invertible (cf. \cite[Theorem III.7.11]{H2}). Let $Z=P_1+\dots +P_m$ and $E=P_1+\dots +P_{m-1}$. Condition i) implies $H^1(C, \mathcal{O}_{C}(k)(-Z))=H^1(C, \mathcal{O}_{C}(k)(-E))=0$ (cf. \cite[Theorem 1.5]{G}). Condition ii) implies $H^0(\mathbb{P}^n,\mathcal{O}_{\mathbb{P}^n}(k))\to H^0(C,\mathcal{O}_C(k))$ is surjective (cf. \cite[Corollary p.492]{GLP}). So the hypotheses of Proposition \ref{LGreco} are satisfied and therefore the thesis follows.
\end{proof}
\medskip

\section{Sets of points lying on curves of degree $2$ and degree $3$}

We recall that Lopez and Pirola proved in \cite[Lemma 2.5]{LP} that if a set of distinct points in $\mathbb{P}^n$ with $n\in\{2,3\}$ is $CB(k)$ and has cardinality at most $3k-1$, then it lies on a curve of degree $2$. If, instead, its cardinality is at most $4k-5$, then it lies on a curve of degree $3$. The two theorems we are going to prove in this section generalize this result to any $\mathbb{P}^n$.\\

Let us start with the first case.

\begin{theorem}\label{Th3}
Let $\Gamma = \{p_1,\dots,p_m\}\subset\mathbb{P}^n$ be a set of distinct points that satisfies the Cayley-Bacharach condition with respect to the complete linear system $|\mathcal{O}_{\mathbb{P}^n}(k)|$, with $k\geq 1$. If $m\leq 3k-1$, then $\Gamma$ lies on a reduced curve of degree $2$.
\end{theorem}

\begin{proof}
If $k=1$, then $m\leq 2$ and the theorem is trivial. If $k=2$, then $m\leq \frac{5}{2}\cdot 2 + 1$ so the theorem follows from \cite[Theorem 1.9]{SU} (cf. Remark \ref{SUresult}). If $n=2$, the assertion is included in \cite[Lemma 2.5]{LP}. We can therefore suppose $k\geq 3$ and $n\geq 3$. Let us define
\begin{equation}\label{alpha3}
\alpha:=\text{max number of points of }\Gamma \text{ lying on a same } 2\text{-plane}.
\end{equation}
Obviously $\alpha \geq 3$. Let us fix a plane $H$ that contains $\alpha$ points of $\Gamma$. If $\alpha=m$, this is equivalent to the case $n=2$. Hence we may assume $\alpha<m$ and we have $m-\alpha$ points of $\Gamma$ outside from $H$. Let us denote by $\Gamma'$ the set of these points. The set $\Gamma'$ is $CB(k-1)$ by Proposition \ref{pointsoutofCB} and then $m-\alpha\geq k+1$ by Remark \ref{allineati}. Furthermore, $m-\alpha\leq 3k-1-\alpha\leq 3k-1-3=3(k-1)-1$ and so, by induction on $k$, $\Gamma'$ lies on a reduced curve of degree $2$, i.e. either on a reduced conic or on two skew lines.

If $\Gamma'$ lies on a reduced conic, it lies on a plane. Thus, by definition of $\alpha$, it must be $m-\alpha\leq \alpha$, i.e. $\alpha\geq \frac{m}{2}$. We have $m-\alpha\leq \frac{m}{2}\leq\frac{3k-1}{2}\leq2(k-1)+1$, so $\Gamma'$ lies on a line $\ell$ (by Remark \ref{allineati}). Hence, by Lemma \ref{minumCB}, $\ell$ contains at least $k+1$ points of $\Gamma$ and so $|\Gamma \backslash \Gamma'|\leq m-k-1\leq 3k-1-k-1<2(k-1)+1$. Since by Proposition \ref{pointsoutofCB} also the set $\Gamma \backslash \Gamma'$ is $CB(k-1)$, it follows that $\Gamma \backslash \Gamma'$ lies on a line $\ell'$, too. So the whole $\Gamma$ lies on $\ell\cup\ell'$.

Let us suppose now that $\Gamma'$ lies on two skew lines $\ell_1$ and $\ell_2$. If one of the two lines does not intersect $\Gamma'$, then $\Gamma'$ lies on a plane and we argue as in the previous case. Moreover, since $\Gamma'$ is $CB(k-1)$, then $\Gamma_1=\Gamma'\cap\ell_1$ and $\Gamma_2=\Gamma'\cap\ell_2$ are $CB(k-1)$ (always by Proposition \ref{pointsoutofCB}). Thus Remark \ref{allineati} yields $q_i:=|\Gamma_i|\geq k$ for $i=1,2$. But, by definition, $\alpha>\max\{q_1,q_2\}\geq k$. Therefore we have $3k<\alpha+q_1+q_2=m\leq 3k-1$, which is impossible. Hence this configuration does not occur.
\end{proof}

By a similar argument, but with more cases to analyse, we get the the second result.

\begin{theorem}\label{Th4}
Let $\Gamma = \{p_1,\dots,p_m\}\subset\mathbb{P}^n$ be a set of distinct points that satisfies the Cyaley-Bacharach condition with respect to the complete linear system $|\mathcal{O}_{\mathbb{P}^n}(k)|$, with $k\geq 1$. If $m\leq 4k-5$, then $\Gamma$ lies on a reduced curve of degree $3$.
\end{theorem}

\begin{proof}
 If $n=2$ or $n=3$ the theorem is proved in \cite[Lemma 2.5]{LP}. Moreover, if $k\leq 4$ we have $m\leq 4k-5\leq 3k-1$ so the theorem follows by Theorem \ref{Th3}. Let us suppose $n\geq 4$ and $k\geq 5$ and let us define
\begin{equation}\label{alpha4}
\alpha:=\text{max number of points of }\Gamma \text{ lying on a same } 3\text{-plane}.
\end{equation}
Obviously $\alpha \geq 4$. Let us fix a linear space $H$ of dimension $3$ that contains $\alpha$ points of $\Gamma$. If $\alpha=m$ we conclude by the assertion for $n=3$. So we assume that $\alpha< m$ and we have $m-\alpha$ points of $\Gamma$ not lying on $H$. Let us denote by $\Gamma'$ the set of these points. The set $\Gamma'$ is $CB(k-1)$ by Proposition \ref{pointsoutofCB} and then $m-\alpha\geq k+1$ by Remark \ref{allineati}. Furthermore, $m-\alpha\leq 4k-5-\alpha\leq 4(k-1)-5$ and so, by induction on $k$, $\Gamma'$ lies on a reduced curve of degree $3$. We have the following three possibilities.

(a) The set $\Gamma'$ lies on a irreducible cubic curve, so in particular it lies on a $\mathbb{P}^3$. Thus, by definition of $\alpha$, it must be $m-\alpha\leq \alpha$ , i.e. $\alpha\geq \frac{m}{2}$. We have $m-\alpha\leq \frac{m}{2}\leq 2k-\frac{5}{2}<2(k-1)+1$, so $\Gamma'$ lies on a line $\ell$ (by Remark \ref{allineati}). Thus, by Lemma \ref{minumCB}, $\ell$ contains at least $k+1$ points of $\Gamma$ and so $|\Gamma \backslash \Gamma'|\leq m-k-1\leq 4k-5-k-1<3(k-1)-1$. Since by Proposition \ref{pointsoutofCB} also the set $\Gamma \backslash \Gamma'$ is $CB(k-1)$, it follows that $\Gamma \backslash \Gamma'$ lies on a reduced curve $C_2$ of degree $2$ by Theorem \ref{Th3}. So the whole $\Gamma$ lies on $C_2\cup\ell $, a curve of degree $3$.

(b) The set $\Gamma'$ lies on $C\cup\ell$, where $C$ is an irreducible conic and $\ell$ is a line. If either $\Gamma'\cap \ell=\emptyset$ or $\Gamma'\cap C=\emptyset$, then the set $\Gamma'$ lies on a $3$-plane and we are in the case $n=3$. The same holds if $C\cap\ell\neq\emptyset$. Thus we assume that $q_1:=|\Gamma'\cup\ell|\neq 0$, $q_2:=|\Gamma'\cap C|\neq 0$ and $q_1+q_2=m-\alpha$. By definition $\alpha\geq\max\{q_1+2,q_2+1\}$. Moreover, since $\Gamma'$ is $CB(k-1)$, $\Gamma'\cap\ell$ and $\Gamma'\cap C$ are both $CB(k-2)$ (by Proposition \ref{pointsoutofCB}) and thus $q_i\geq k$ for $i=1,2$, we have that $m-\alpha=q_1+q_2\geq 2k$. Therefore $\alpha\leq m-2k\leq 4k-5-2k<2(k-2)+1$ and, by Remark \ref{allineati}, it follows that $\Gamma\backslash \Gamma'$ lies on a line. But we have also that $\alpha\geq k+2$ and then $q_2=m-\alpha-q_1\leq 4k-5-k-2-k<2(k-2)+1$. Thus also $\Gamma'\cap C$ lies on a line. In conclusion, the set $\Gamma$ lies on three lines.

(c) The set $\Gamma'$ lies on the union of three distinct lines $\ell_1,\ell_2,\ell_3$. If $\Gamma'$ does not intersect one of the lines $\ell_i$, then $\Gamma'$ lies on a $3$-plane and we are in the case $n=3$. The same holds if any $\ell_i$ intersects at least one of the others two lines. Thus $q_i:=|\Gamma'\cap \ell_i|\neq 0$, for $i=1,2,3$, and $m-\alpha\in\{q_1+q_2+q_3, q_1+q_2+q_3-1\}$. Moreover, any $\Gamma'\cap \ell_i$ is $CB(k-3)$ by Proposition \ref{pointsoutofCB}, so $q_i\geq k-1$ for any $i=1,2,3$. But, by definition, $\alpha>q_i$  for any $i=1,2,3$, thus $\alpha>k-1$. Then we have $4k-5<\alpha+q_1+q_2+q_3-1\leq m\leq 4k-5$, a contradiction. Hence configuration (c) does not occur.
\end{proof}
\bigskip

\section{Sets of points lying on curves of degree $4$}

Aim of this section is to prove the following theorem, that is the main result of this paper.

\begin{theorem}\label{Th5}
Let $\Gamma = \{p_1,\dots,p_m\}\subset\mathbb{P}^n$ be a set of distinct points that satisfies the Caley-Bacharach condition with respect to the complete linear system $|\mathcal{O}_{\mathbb{P}^n}(k)|$, with $k\geq 1$. If $m\leq 5k-11$, then $\Gamma$ lies on a reduced curve of degree $4$.
\end{theorem}

This theorem generalizes to any $\mathbb{P}^n$ the analogue result that Lopez and Pirola proved in \cite[Lemma 2.5]{LP} just in $\mathbb{P}^2$. We point out that, in contrast with the theorems of the previous section, in this case the assertion have to be proved even in $\mathbb{P}^3$.

\begin{remark}\label{katleast7}
If $k<7$, we have $5k-11<4k-5$, so, by Theorem \ref{Th4}, $\Gamma$ lies on a reduced curve of degree $3$. Moreover, for $k=7$ we have $4\cdot 7-5=23$ and $5\cdot 7-11=24$. Thus for $k=7$ Theorem \ref{Th5} has to be proved only in the case $m=24$.
\end{remark}

Based on the previous remark, we can just focus on the case $k\geq 7$.\\

In order to prove Theorem \ref{Th5}, we need to deal with some particular configurations for which the general approach does not work. Let $\Gamma = \{p_1,\dots,p_m\}\subset\mathbb{P}^n$ be as in Theorem \ref{Th5} and let us define
\begin{equation}\label{alpha}
\alpha:=\text{max number of points of }\Gamma \text{ lying on a same hyperplane}.
\end{equation}

In the following subsection we prove Theorem \ref{Th5} in $\mathbb{P}^3$ and $\mathbb{P}^4$ and for values of $\alpha$ at most $4$. In the next one we complete the proof for all the remaining cases.

\subsection{Lower dimensions and $\alpha\leq 4$}

Goal of this subsection is proving the following proposition.

\begin{proposition}\label{Th5n34}
Theorem \ref{Th5} holds in $\mathbb{P}^3$ and $\mathbb{P}^4$ when $\alpha$, defined as in (\ref{alpha}), is at most $4$.
\end{proposition}

We start with some remarks.

\begin{remark}\label{pointsonline}
Condition $\alpha\leq 4$ implies that in $\mathbb{P}^3$ we can have at most $3$ points of $\Gamma$ on a line and at most $4$ points of $\Gamma$ on a plane. Instead, in $\mathbb{P}^4$ this condition implies that we can have at most $2$ points of $\Gamma$ on a line, at most $3$ points of $\Gamma$ on a plane and at most $4$ points of $\Gamma$ on a $3$-space.
\end{remark}

\begin{remark}\label{notCB}
Since by Lemma \ref{minumCB} a set of points $CB(k)$ must have cardinality at least $k+2$, it follows by Remark \ref{pointsonline} that, if  $\alpha\leq 4$, on a line in $\mathbb{P}^3$ we can have only a subset $CB(1)$ of points of $\Gamma$, while on a plane we can have only a subset $CB(2)$ of points of $\Gamma$. For the same reasons, in $\mathbb{P}^4$ points of $\Gamma$ on a line can not satisfy  any Caley-Bacharach condition, they can be only $CB(1)$ on a plane and at most $CB(2)$ on a $3$-space. Thus in $\mathbb{P}^3$ points of $\Gamma$ on a line can not be $CB(k-s)$ with $s\leq 5$ (if $k\geq 7$) or $s\leq 6$ (if $k\geq 8$) and points of $\Gamma$ on a plane can not be $CB(k-s)$ with $s\leq 4$ (if $k\geq 7$) or $s\leq 5$ (if $k\geq 8$). Similarly, if $k\geq 7$, in $\mathbb{P}^4$ points of $\Gamma$ on a line con not be $CB(k-s)$ with $s\leq 6$, points of $\Gamma$ on a plane can not be $CB(k-s)$ with $s\leq 5$ and points of $\Gamma$ on a $3$-space can not be $CB(k-s)$ with $s\leq 4$.
\end{remark}

\begin{remark}\label{CBoutcurve}
Corollary in \cite[p.492]{GLP} ensures that an integral non-degenerate curve of degree $d$ in $\mathbb{P}^n$ is cut  out by hypersurfaces of degree $d+2-n$. As a consequence, if a set of points $A$ is $CB(k)$ and a proper subset $B\subset A$ lies on an integral non-degenerate curve of degree $d$, the set $A\backslash B$ is $CB(k+n-d-2)$. In particular $A\backslash B$ is $CB(k-d+1)$ in $\mathbb{P}^3$ and $CB(k-d+2)$ in $\mathbb{P}^4$.
\end{remark}

We prove Proposition \ref{Th5n34} by a series of lemmas. For sake of clarity, we recall that $\Gamma = \{p_1,\dots,p_m\}\subset\mathbb{P}^n$ is a set of distinct points that satisfies the Caley-Bacharach condition with respect to the complete linear system $|\mathcal{O}_{\mathbb{P}^n}(k)|$, $m\leq 5k-11$ and $\alpha$, defined in (\ref{alpha}), is at most $4$. Moreover, we focus on the case $k\geq 7$ (see Remark \ref{katleast7}).

\begin{lemma}\label{curve9}
Under the hypotheses of Proposition \ref{Th5n34}, $\Gamma\subset\mathbb{P}^n$ lies on a complete intersection curve $C$ of degree $d$, where
\begin{enumerate}
\item[a)] $d=8$ if $n=3$ and $k=7$,
\item[b)] $d=9$ if $n=3$ and $k\geq 8$,
\item[c)] $d=8$ if $n=4$ and $k\geq 7$.
\end{enumerate}
\end{lemma}

\begin{proof} Case $a)$. By Remark \ref{katleast7}, we have $m=24$. Let us fix $9$ points of $\Gamma$ and let us consider a quadric $Q_2$ passing through them. We claim that $Q_2$ passes through all the point of $\Gamma$. If not, the points of $\Gamma$ outside from $Q_2$ would be $CB(5)$ and then they would be at least $7$. Now, if they are at most $14$, by Theorem \ref{Th3}, they must lie on a curve of degree $2$ (i.e. they lie either on a plane or on two skew lines), but this is not possible by Remark \ref{pointsonline}. Otherwise they must be exactly $15$ and then they must lie on a cubic curve $C_3$ by Theorem \ref{Th4}. Moreover, the curve $C_3$ must be irreducible and non-degenerate, again by Remark \ref{pointsonline}. If $C_3$ passes trough all the points of $\Gamma$, then Proposition \ref{Th5n34} follows. Otherwise, the points outside from $C_3$ should be $CB(5)$ by Remark \ref{CBoutcurve}, hence they must be at least $7$ by Lemma \ref{minumCB} and must lie on a curve of degree $2$ by Theorem \ref{Th3}; this is not possible, once again by Remark \ref{pointsonline}. Thus $Q_2$ passes through all the points of $\Gamma$, as claimed. Now, by dimensional reasons, we can consider a quartic $Q_4$ passing through all the points of $\Gamma$ and not containing $Q_2$. The complete intersection $Q_2\cap Q_4$ is a curve of degree $8$ containing  $\Gamma$.

Case $b)$. We fix $18$ points of $\Gamma$ and we consider two independent cubics passing through them. Their complete intersection is a curve of degree $9$ that we denote by $C$. Let $\Gamma_C=\Gamma \cap C$. Clearly $\Gamma=\Gamma_C \cup \Gamma'$ with $\Gamma'\cap C=\emptyset$. We claim that $\Gamma'=\emptyset$. In fact, $|\Gamma'|\leq 5k-23<5(k-3)-11$ and $\Gamma'$, if non-empty, is $CB(k-3)$ and hence it has cardinality at least $7$ (so it can not lie on a plane by Remark \ref{pointsonline}). By induction on $k$, $\Gamma'$ lies on a curve $C_4$ of degree $4$. If $C_4$ passes through all the points of $\Gamma$, then Proposition \ref{Th5n34} follows.  Let us therefore suppose $C_4$ dose not pass through all the points of $\Gamma$. If $C_4$ is irreducible, then the subset $\Gamma'_C$ of points of $\Gamma$ not on $C_4$ is $CB(k-3)$ by Remark \ref{CBoutcurve}. But not both the sets $\Gamma'_C$ and $\Gamma'$ can have cardinality greater than $3(k-3)$ (otherwise $6(k-3)>5k-11$ for $k\geq 8$). So, by Theorem \ref{Th3}, at least one between $\Gamma'$ and $\Gamma'_C$ must lie either on a conic or on two skew lines. By Remark \ref{notCB}, this is impossible for a set that is $CB(r)$ with $r\geq 5$. Thus  $C_4$ must be reducible. But in this case, by Proposition \ref{pointsoutofCB}, the points of $\Gamma'$ on any line in the decomposition would be at least $CB(k-6)$ and those on any conic at least $CB(k-5)$; impossible in any case. So we can have just a non-empty irreducible component $C_3$ of degree $3$. But this is also impossible since the argument used for the case $C_4$ irreducible works a fortiori for the case $C_3$ irreducible. In conclusion, $\Gamma'=\emptyset$ as wanted and therefore $\Gamma \subset C$.

Case $c)$. Let us fix $12$ points of $\Gamma$ and let us consider three independent quadrics passing through them. Their complete intersection is a curve of degree $8$ that we denote by $C$. As before, let $\Gamma_C=\Gamma \cap C$. Clearly $\Gamma=\Gamma_C \cup \Gamma'$, with $\Gamma'\cap C=\emptyset$. Let us show that $\Gamma'=\emptyset$. In fact, $|\Gamma'|\leq 5k-13=5(k-2)-11$ and in this case $\Gamma'$, if non-empty, is $CB(k-2)$. So, by induction on $k$, $\Gamma'$ lies on a curve $C_4$ of degree $4$. We can now follow the same argument used in point $b)$, since (although $k$ may be $7$)  the sets $\Gamma'_C$ and $\Gamma'$ are now $CB(k-2)$ and in $\mathbb{P}^4$ we can have at most $2$ or $3$ points on lines and planes respectively. Then also in this case $\Gamma'=\emptyset$, which implies $\Gamma \subset C$.\end{proof}

For the next lemmas, we need the following remark.

\begin{remark}\label{confrontopa}
In $\mathbb{P}^3$ the \emph{Castelnuovo bound} (see e.g. \cite[Theorem 3.7]{Ha}) becomes $p_a\leq m(m-1)+m\varepsilon$ with $d-1=2m+\varepsilon$ and $\varepsilon=0,1$. In any case we get
\begin{equation}\label{pa3}
p_a\leq\frac{d^2-4d+4}{4}.
\end{equation}
In $\mathbb{P}^4$ the bound becomes $p_a\leq \frac{3}{2}m(m-1)+m\varepsilon$ with $d-1=3m+\varepsilon$ and  $\varepsilon=0,1,2$, so in this case we get
\begin{equation}\label{pa4}
p_a\leq\frac{d^2-5d+6}{6}.
\end{equation}
Moreover, we note that, for a fixed degree $d$, the maximum $p_a$ in $\mathbb{P}^4$ is smaller than the maximum $p_a$ in  $\mathbb{P}^3$.\\
\end{remark}

The following lemma shows that the curve $C$ of Lemma \ref{curve9} can not be irreducible.

\begin{lemma}\label{notgeq5}
Under the hypotheses of Proposition \ref{Th5n34}, $\Gamma$ can not lie on an irreducible curve $C$ of degree $d\geq 5$.
\end{lemma}

\begin{proof} This follows from Corollary \ref{CGreco}. Let us suppose, by contradiction, that $\Gamma$ lies on an irreducible curve $C$ of degree $d\geq 5$. Since $C$ is irreducible we can consider it is also reduced, otherwise any reduced component would have degree at most $4$ and Proposition \ref{Th5n34} would follow. So the curve $C$ is integral. Condition $ii)$ of Corollary \ref{CGreco} is clearly verified. For condition $i)$ we have in any case (see relations (\ref{pa3}) and (\ref{pa4}) in Remark \ref{confrontopa})
$$\frac{m+p_a-2}{d}\leq \frac{5k-11 + \frac{d^2-4d+4}{4}-2}{d}=\frac{10k-22+d^2-4d}{2d}.$$
Let us check when the last term is less than $k$. This happens when $2k(d-5)>d^2-4d-22$, which is verified for any $d$ between $5$ and $9$ (for any $k\geq 3$). Therefore Corollary \ref{CGreco} leads us to the absurd conclusion that the set $\Gamma$ can not be $CB(k)$. \end{proof}

The curve $C$ containing $\Gamma$ must therefore be reducible. Some components of $C$ could not intersect $\Gamma$. We point out that, if the sum of the degrees of the components with non-empty intersection with $\Gamma$ is lower than or equal to $4$, then we get the thesis of Proposition \ref{Th5n34}. In the following lemmas we examinate the other cases. Let us  start with a remark.

\begin{remark}\label{pointsreorg}
If some points of $\Gamma$ are contained in the union of $4$ lines $\ell_1,\ell_2,\ell_3,\ell_4$, then these points lie on $3$ planes. Indeed, since $\ell_4$ contains at most three points of $\Gamma$ (see Remark \ref{pointsonline}), then there exists $3$ planes $\pi_1,\pi_2,\pi_3$ containing respectively $\ell_1,\ell_2,\ell_3$ and one of the points of $\ell_4$.
\end{remark}

\begin{lemma}\label{CBcomp5}
Under the hypotheses of Proposition \ref{Th5n34}, let $C$ be the curve of Lemma \ref{curve9}, and suppose $C=\widetilde{C}\cup \bigcup_{i\in I} C_i$ where $\widetilde{C}$ and the $C_i$'s are irreducible and the component $\widetilde{C}$ has degree $\widetilde{d}\geq 5$. Let $C'=\widetilde{C}\cup \bigcup_{i\in I'} C_i$ where $I'$ indexes the components $C_i$ of $C$ with non-empty intersection with $\Gamma$. Then the set $\widetilde{\Gamma}=\Gamma \cap \widetilde{C}$ is $CB(r)$ with 
\begin{itemize}
\item[(i)] $r\geq k-3$ if $\widetilde{d}=5$,
\item[(ii)] $r\geq k-3$ if $\widetilde{d}=6$ and $\bigcup_i C_i$ is the union of three lines,
\item[(iii)] $r\geq k-2$ otherwise. 
\end{itemize}
\end{lemma}

\begin{proof} First of all it is $5\leq \widetilde{d}\leq 8$ since $C$ is reducible by Lemma \ref{notgeq5}. Moreover, we have $|I'|\leq 4$ and $\sum_{i\in I'} d_i\leq 4$ with $d_i:=$deg$C_i$.

If $\widetilde{d}=6$ and $\bigcup_i C_i$ is the union of three lines then $\widetilde{\Gamma}$ is at least $CB(k-3)$ by Proposition \ref{pointsoutofCB}, so we get case $(ii)$.

It follows by the same proposition that if, instead, $\bigcup_{i\in I'} C_i$ is one line, one irreducible conic or two lines (which are the only possibilities if $\widetilde{d}\in\{7,8\}$) then $\widetilde{\Gamma}$ is at least $CB(k-2)$. Moreover, the same is true either if $\bigcup_i C_i$ is one line and an irreducible conic or, by Remark \ref{CBoutcurve}, if it is just an irreducible cubic curve. This proves case $(iii)$.

For $\widetilde{d}=5$ the situation is more complicated. However, in $\mathbb{P}^4$ (where $\widetilde{d}+\sum_i d_i \leq 8$) the possibilities for $\sum_i C_i$ are the same of those in the case $\widetilde{d}=6$. In $\mathbb{P}^3$ it remains to consider the cases in which $\sum_i d_i=4$. If  $\bigcup_{i\in I'} C_i$ is either two irreducible conics or one irreducible conic and two lines, we have that $\widetilde{\Gamma}$ is $CB(r)$ with $r\geq k-3$ by Proposition \ref{pointsoutofCB}. Moreover, this is still true if  $\bigcup_{i\in I'} C_i$ is the union of four lines since in this case (see Remark \ref{pointsreorg}) the points of $\Gamma$ on these lines actually lie on three planes. Finally, since when $\bigcup_{i\in I'} C_i$ is just an irreducible cubic curve  $\widetilde{\Gamma}$ is $CB(k-2)$, if we have also a line,  $\widetilde{\Gamma}$ is $CB(k-2)$ by Proposition \ref{pointsoutofCB}. \end{proof}

\begin{lemma}\label{no5atall}
Under the hypotheses of Proposition \ref{Th5n34}, let $C$ be the curve of Lemma \ref{curve9}. Then $C$ can not decompose as in Lemma \ref{CBcomp5}.
\end{lemma}

\begin{proof} We proceed by contradiction supposing that the curve $C$ decomposes as in Lemma \ref{CBcomp5}.  With the same notation of this lemma, in $\mathbb{P}^3$ we have that the set $\Gamma'=\Gamma\backslash \widetilde{\Gamma}$ of points of $\Gamma$ outside from $\widetilde{C}$ is $CB(k-\widetilde{d}+1)$, whereas in $\mathbb{P}^4$ the set $\Gamma'$ is $CB(k-\widetilde{d}+2)$ (see Remark \ref{notCB}). In any case $|\Gamma'|\geq k-\widetilde{d}+3$. Let us analyse the following two cases.

1) If $\widetilde{d}\in \{5,6\}$, then $|\Gamma'|\geq 4$ when $k\geq 7$ and thus $\Gamma'$ can not lie on a line. Then, by Remark \ref{allineati}, $\Gamma'$ must have in fact cardinality at least $2(k-\widetilde{d}+1)+2=2k-\widetilde{d}+4$. This implies $|\widetilde{\Gamma}|\leq 5k-11-2k+\widetilde{d}-4=3k+2\widetilde{d}-15$. By Lemma \ref{CBcomp5} we know that $\widetilde{\Gamma}$ is at least $CB(k-3)$ with $\widetilde{d}=5$. Now we want to use Corollary \ref{CGreco}. Condition $ii)$ is clearly satisfied. Let us verify condition $i)$. We get that
\begin{equation}\label{eq1}
k-3> \frac{3k+2\widetilde{d}-15+\frac{\widetilde{d}^2-4\widetilde{d}+4}{2}-2}{\widetilde{d}} \iff k>\frac{\widetilde{d}^2+6\widetilde{d}-30}{2\widetilde{d}-6}
\end{equation}
and the last one is true for $\widetilde{d}=5$ and for any $k\geq 7$, whereas if $\widetilde{d}=6$ the last relation in (\ref{eq1}) is true only for $k>7$. But if $k=7$  and $\widetilde{d}=6$, the components $C_i$ can not be three lines and so, always by Lemma \ref{CBcomp5}, $\widetilde{C}$ is at least $CB(k-2)$. If we now substitute $k-2$ in the left-hand side of (\ref{eq1}), we get that it holds for $k\geq 6$. Thus in any case we find a contradiction by Corollary \ref{CGreco}.

2) If $\widetilde{d}\in \{7,8\}$, by Lemma \ref{CBcomp5}, $\widetilde{\Gamma}$ is at least $CB(k-2)$. Moreover $|\widetilde{\Gamma}|\leq 5k-11-k+\widetilde{d}-3=4k+\widetilde{d}-14$. So this time we get
$$k-2> \frac{4k+\widetilde{d}-14+\frac{\widetilde{d}^2-4\widetilde{d}+4}{2}-2}{\widetilde{d}} \iff k>\frac{\widetilde{d}^2+2\widetilde{d}-28}{2\widetilde{d}-8}$$
that is true for $\widetilde{d}\in\{7,8\}$ and for any $k\geq 7$. So, again by Corollary \ref{CGreco}, we find a contradiction. \end{proof}

Accordingly, in the decomposition of $C$ can appear only components of degree at most $4$. In the following lemmas we exclude the remaining cases.

\begin{lemma}\label{nolineconics}
Under the hypotheses of Proposition \ref{Th5n34}, let $C$ be the curve of Lemma \ref{curve9}. Then $C$ can not decompose only in lines and irreducible conics.
\end{lemma}

\begin{proof} In $\mathbb{P}^4$, by Remark \ref{pointsonline}, we can have at most $16$ points on lines and on irreducible conics. As $|\Gamma|\geq 24$, this proves the lemma in $\bP^4$.

In $\mathbb{P}^3$, by the same reason, if we denote by $r$ the number of lines and by $c$ the number of irreducible conics, the only possible cases are $(r,c)\in\{(9,0),(8,0),(7,1)\}$. Let us start with the case $(r,c)=(8,0)$. The only configuration allowed is that with $|\Gamma|=24$ and three points on each line. By Remark \ref{pointsreorg}, the points lie on $6$ planes with $4$ points on each of them. If we consider a line passing through two points on one of these planes, the remaining two points on the same plane are $CB(k-6)$ and this is not possible by Remark \ref{notCB}. As for the cases $(r,c)\in\{(9,0),(7,1)\}$, these are possible only for $k\geq 8$. But, always by Remark \ref{pointsreorg}, we have that any line would be $CB(k-6)$ and thus would contain $k-4\geq 4$ points, which is impossible. Hence the lemma follows in $\bP^3$, too.\end{proof}

\begin{lemma}\label{notonlyless2}
Under the hypotheses of Proposition \ref{Th5n34}, let $C$ be the curve of Lemma \ref{curve9}. Then $C$ can not decompose in $C=\widetilde{C}\cup\bigcup_{i\in I} C_i$ with $\widetilde{C}$ and the components $C_i$'s irreducible, $\widetilde{d}:=$deg$\widetilde{C}\in\{3,4\}$, $d_i:=$deg$C_i\leq 2$ for any $i\in I$ and $\Gamma \cap \bigcup_i C_i \neq \emptyset$.
\end{lemma}

\begin{proof} Let us start analysing the situation in $\mathbb{P}^4$. An irreducible curve of degree $3$ lies on a $3$-space, which can contain at most $4$ points of $\Gamma$ in this configuration (see Remark \ref{pointsonline}). So, if $\widetilde{d}=3$ we could have at most $14$ points of $\Gamma$, in contrast with the assumption $|\Gamma|\geq 24$. If, instead, $\widetilde{d}=4$ then the set $\Gamma \cap \bigcup_i C_i$ is $CB(k-2)$ by Remark \ref{CBoutcurve}, so the points of $\Gamma$ on any $C_i$ have to be at least $CB(k-5)$ by Proposition \ref{pointsoutofCB} and hence contains at least $k-3\geq 4$ points of $\Gamma$ by Lemma \ref{minumCB}. But this is not possible since on a plane we can have at most three points of $\Gamma$.

In $\mathbb{P}^3$ the components $C_i$ can not contain irreducible conics. In order to see this, we shall use Remark \ref{notCB} consistently. If $\bigcup_i C_i$ were only irreducible conics, arguing as usual, the points of $\Gamma$ on any of these conics would be at least $CB(k-4)$. In the same way, one sees that if in $\bigcup_i C_i$  there appear a line and at least an irreducible conic, then the points of $\Gamma$ on the line would be at least $CB(k-6)$ if $k\geq 8$ (since this is the only case in which $C$ can have degree $9$) and at least $CB(k-5)$ if $k=7$. Thus the $C_i$'s must be all lines.

If deg$\widetilde{C}=3$, then the points of $\Gamma$ on $\bigcup_i C_i$ are $CB(k-2)$ by Remark \ref{CBoutcurve}. Moreover, remembering Remark \ref{pointsreorg}, if $|I|\leq 5$ it follows that any $\Gamma \cap C_i$ is at least $CB(k-5)$ by Proposition \ref{pointsoutofCB} and this is not possible by Remark \ref{notCB}. But the same argument works for the case $|I|=6$. In fact, this configuration is possible only for $k\geq 8$ (degree of $C$ equal to $9$) and in this case any set $\Gamma \cap C_i$ would be at least $CB(k-6)$, which is still impossible by Remark \ref{notCB}.

If deg$\widetilde{C}=4$, then the points of $\Gamma$ on $\bigcup_i C_i$ are $CB(k-3)$ by Remark \ref{CBoutcurve}. Now, if $|I|=5$ (which is possible only in the case $k\geq 8$), by Remark \ref{pointsreorg} and Proposition \ref{pointsoutofCB} we would again conclude that $\Gamma \cap C_i$ is $CB(k-6)$. Whereas, if $|I|\leq 3$, we would get directly by Proposition \ref{pointsoutofCB} that $\Gamma \cap C_i$ is $CB(k-5)$. So these cases are impossible and it remains to study only the case of four lines. By the same reasons of above, the points of $\Gamma$ on each of these lines are $CB(k-6)$. This implies that $k=7$ and that on any of these lines we must have at least $3$ points of $\Gamma$, that is, exactly $3$, since on a line we can not have more than $3$ points of $\Gamma$. Thus $|\widetilde{\Gamma}|=12$. Now, the set $\widetilde{\Gamma}$ is $CB(3)$ and we can use Corollary \ref{CGreco} to find a contradiction.  Indeed, $3>2=4-2$ (condition $ii)$ of Corollary \ref{CGreco}) and from (\ref{pa3}) we get $p_a\leq 1$, so $3>\frac{11}{4}=\frac{12+1-2}{4}$ and condition $i)$ of Corollary \ref{CGreco} is verified, too. \end{proof}

\begin{lemma}\label{notwo}
Under the hypotheses of Proposition \ref{Th5n34}, let $C$ be the curve of Lemma \ref{curve9}. Then in the decomposition of  $C$ can not appear two curves of degree $4$ or two curves of degree $3$ containing points of $\Gamma$.
\end{lemma}

\begin{proof} Let us start with the case of two curves of degree $4$, $C_1$ and $C_2$. By Remark \ref{CBoutcurve}, in $\mathbb{P}^4$ the sets $C_i\cap\Gamma$, $i=1,2$, are $CB(k-2)$. Moreover, not both of theme can have cardinality greater than or equal to $3(k-2)$, otherwise for any $k\geq 7$ we would have $m\geq 6(k-2)>5k-11 \geq |\Gamma|$. Then, by Theorem \ref{Th3}, at least one among $C_1\cap\Gamma$ and $C_1\cap\Gamma$ must lie on a curve of degree $2$, but this is not possible by Lemma \ref{notonlyless2}. In $\mathbb{P}^3$, if $k\geq 8$, in the decomposition of $C$ it can appear also a line $\ell$. But the set $\Gamma\cap \ell$ would be $CB(k-6)$ and this is impossible by Remark \ref{notCB}. So we may have only the two curves of degree $4$, $C_1$ and $C_2$, and the sets $C_i\cap\Gamma$, $i=1,2$, are $CB(k-3)$. If $k\geq 8$ these two sets can not have both cardinality greater than or equal to $3(k-3)$; otherwise we would have $m\geq 6(k-3)>5k-11$ and, as before, we would conclude that at least one of these sets lies on a curve of degree $2$. If $k=7$ ($m=24$), we can find again a contradiction by Corollary \ref{CGreco}. The computation is the same as the one at the end of Lemma \ref{notonlyless2}, since now $k-3=4>3$ and at least one of the two curves must contains at most $12$ points of $\Gamma$.

Let us consider the case of two irreducible curves of degree $3$, $D_1$ and $D_2$. It is easy to see that in $\mathbb{P}^4$ this is not possible. Indeed, on a such curve we can have at most $4$ points of $\Gamma$, so $|\Gamma|$ would be at most $8$.

In order to deal with this configuration in $\mathbb{P}^3$, let us consider the set $\Gamma\backslash (D_1\cup D_2)$ of points of $\Gamma$ not on these two curves. If $\Gamma\backslash (D_1\cup D_2)\neq\emptyset$, then it is $CB(k-4)$ and so it must contain at least $5$ points. Thus it can not lie neither on a line nor on an irreducible conic. Furthermore, it can not lie even on two lines, otherwise the set of points of $\Gamma$ on each of them would be $CB(k-5)$. If $k\geq 8$ (that is, $\deg C=9$) we could have also the following further possibilities: (a) three lines $\ell_1,\ell_2,\ell_3$; (b) a line $\ell$ and an irreducible conic $C_2$; (c) another irreducible curve $D_3$ of degree $3$. 

Case (a) does not occur since, by Proposition \ref{pointsoutofCB}, any set $\Gamma\cap \ell_i$, for $i=1,2,3$, would be $CB(k-6)$. Also case (b) does not occur by the same reason. Namely, the set $\Gamma\cap \ell$ would be $CB(k-5)$. As for case (c), we note that the sets $\Gamma\cap D_3$, $\Gamma\cap D_3$ and $\Gamma\cap D_3$ are $CB(k-4)$ by Remark \ref{CBoutcurve} and not all of them can have cardinality greater then or equal to $2(k-4)+2$ (since, otherwise, $|\Gamma|\geq 6(k-4)+6>5k-11\geq |\Gamma|$). Thus, by Proposition \ref{allineati}, at least one of these sets would lie on a line and we would have the configuration of one of the previous cases.

Finally, let us suppose $\Gamma\backslash (D_1\cup D_2)=\emptyset$, i.e. the non-empty components of $C$ (with respect to the points of $\Gamma$) are only the two irreducible curves of degree 3. We can prove this case does not occur by the same argument in case $(c)$. More precisely, the sets $\Gamma\cap D_i$, for $i=1,2$, are $CB(k-2)$ and, arguing as before, we see that at least one of these sets must lie on a curve of lower degree. But this is not possible by Lemma \ref{notonlyless2}. \end{proof}

\begin{lemma}\label{not34}
Under the hypotheses of Proposition \ref{Th5n34}, let $C$ be the curve of Lemma \ref{curve9}. Then in the decomposition of $C$ can not appear simultaneously an irreducible curve $C_3$ of degree $3$ and an irreducible curve $C_4$ of degree $4$ containing points of $\Gamma$.
\end{lemma}

\begin{proof} This configuration does not occur in $\mathbb{P}^4$ since, by Remark \ref{CBoutcurve} and Proposition \ref{pointsoutofCB}, the set $\Gamma\cap C_3$ would be at least $CB(k-3)$, which is impossible by Remark \ref{notCB}.

In $\mathbb{P}^3$ with $k=7$ the curve $C$ has degree $8$ (see Lemma \ref{curve9}). Let us suppose that $C=C_3\cup C_4 \cup \ell$, with $\ell$ a line. If $\Gamma\cap \ell\neq\emptyset$, then this set would be $CB(k-5)$ by Remark \ref{CBoutcurve}, but this is not possible by Remark \ref{notCB}. If $k\geq 8$, then the curve $C$ has degree $9$. If in the decomposition of $C$ it appears an irreducible conic that intersects $\Gamma$, then the same argument made for case $k=7$ shows that this intersection must be $CB(k-5)$; again impossible by Remark \ref{notCB}.  If, instead, in the decomposition of $C$ there appear two skew lines, $\ell_1$ and $\ell_2$, intersecting $\Gamma$, then each set $\Gamma\cap \ell_i$, for $i=1,2$, would be $CB(k-6)$; once again impossible by Remark \ref{notCB}.

So we can have only points of $\Gamma$ on $C_3$ and on $C_4$. It follows that the points on $C_3$ are $CB(k-3)$ and that the points on $C_4$ are $CB(k-2)$. These sets of points can not have at the same time cardinality greater than or equal to $3(k-3)$ and $4(k-3)$ respectively, otherwise their sum would overcome the maximal cardinality of $\Gamma$. Thus, either the points on $C_3$ lie on a curve of degree lower that $3$, or the points on $C_4$ lie on a curve of degree lower that $4$. Both these situations are not allowed by one of the previous lemmas. \end{proof}

Putting all these lemmas together, we can now prove Proposition \ref{Th5n34}.

\begin{proof}[Proof of Proposition \ref{Th5n34}] Lemma \ref{curve9} ensures that $\Gamma$ lies on a curve $C$ of degree $8$ or $9$, but Lemmas \ref{notgeq5} and \ref{no5atall} imply that $\Gamma$ can not intersect $C$ in an irreducible component of degree greater than $4$ (in particular, $C$ can not be irreducible). The components of $C$ containing points of $\Gamma$ can not be only lines and conics by Lemma \ref{nolineconics}. So we have that among these components it must appear at least one of degree $3$ or $4$. Actually, by Lemmas \ref{notonlyless2}, \ref{notwo} and \ref{not34}, $\Gamma$ is contained in exactly one of these components. \end{proof}

\subsection{Proof of the main theorem}

We can finally prove Theorem \ref{Th5}.

\begin{proof}[Proof of Theorem \ref{Th5}] If $n=2$ the assertion follows by \cite[Lemma 2.5]{LP}. The cases $n=3$ and $n=4$ with $\alpha\leq 4$, where $\alpha$ is defined as in (\ref{alpha}), are proved in Proposition \ref{Th5n34}. Let us suppose now $n\geq 3$ and let us define $\alpha$ as in (\ref{alpha}) if $n\in\{3,4\}$ (i.e. the maximum number of points of $\Gamma$ lying on a hyperplane). If $n\geq 5$, let us define $\alpha$ as
\begin{equation}\label{alpha5}
\alpha:=\text{max number of points of }\Gamma \text{ lying on a same } 4\text{-plane}.
\end{equation}
In this last case obviously $\alpha \geq 5$, but we can now assume $\alpha \geq 5$ even in the cases $n\in\{3,4\}$.

Let us fix a linear subspace $H$, of the right dimension, that contains $\alpha$ points of $\Gamma$ and let $\Gamma_H=\Gamma\cap H$. If $\alpha=m$ we are in one of the cases already dealt with. We can therefore assume $\alpha<m$ and then we have $m-\alpha$ points of $\Gamma$ outside from $H$. Let us denote by $\Gamma'$ the set of these points. The set $\Gamma'$ is $CB(k-1)$ by Proposition \ref{pointsoutofCB} and then $m-\alpha\geq k+1$ by Remark \ref{allineati}. Furthermore, $m-\alpha\leq 5k-11-\alpha\leq 5(k-1)-11$ and so, by induction on $k$, $\Gamma'$ lies on a reduced curve $C_4$ of degree $4$.

We claim that $\Gamma_H$ is $CB(k-s)$ with $s\leq 4$. Indeed, if $C_4$ is irreducible, then it is contained in a linear subspace of dimension $2\leq t \leq 4$. By Remark \ref{CBoutcurve}, the curve $C_4$ is cut out by hypersurfaces of degree $6-t\leq 4$, thus $\Gamma_H$ is $CB(k+t-6)$. If the curve $C_4$ decomposes in an irreducible curve $C_3$ of degree $3$ and a line $\ell$, by the same reason,  the set $\Gamma_H$ is either $CB(k+t-5)$, with $2\leq t\leq 3$, if $\Gamma'$ does not intersect $\ell$, or $CB(k+t-6)$ otherwise. Finally, if $C_4$ decomposes in lines and conics, then the set $\Gamma'$ lies at most on $4$ distinct linear subspaces and hence, by Proposition \ref{pointsoutofCB}, $\Gamma_H$ is at least $CB(k-4)$. This proves the claim.

Now, if $m-\alpha\geq 3(k-1)$, then $\alpha=m-(m-\alpha)\leq 2k-8 < 2(k-s)+1$ with $s\leq 4$, thus $\Gamma_H$  lies on a line by Remark \ref{allineati}. But, by the very definition of $\alpha$, it must be $m=\alpha$ and this is impossible since we have points of $\Gamma$ outside from $H$. If, instead, $m-\alpha\leq 3(k-1)-1$, by Theorem \ref{Th3}, the set $\Gamma'$ lies on a reduced curve of degree $2$ and thus, by Proposition \ref{pointsoutofCB}, $\Gamma_H$ is $CB(k-1)$ or $C(k-2)$, depending on whether $\Gamma'$ lies on one or two linear subspaces. If also $\Gamma_H$ lies on a curve of degree $2$ the theorem is proved. Otherwise, by Theorem \ref{Th3}, it must be at least $\alpha\geq 3(k-2)$ and then $m-\alpha\leq 2k-5<2(k-1)+1$, so $\Gamma'$ lies on a line. This implies that $\Gamma_H$ is certainly $CB(k-1)$ and it must lie on a curve of degree $3$, because otherwise, by Theorem \ref{Th4}, we would have $\alpha\geq 4(k-1)-4$ and consequently $m-\alpha\leq k-3$; this is impossible since we pointed out that $m-\alpha\geq k+1$. Thus $\Gamma$ lies on a curve of degree $4$. \end{proof}

\begin{remark}
As \cite[Theorem 1.9]{SU}, compared with Theorem \ref{Th3}, shows, bound (\ref{boundLP}) for $h=3$ is not sharp, at least not for all the values of $k$. In fact, we already noted that for $k\in\{1,2\}$ the bound $\frac{5}{2}k+1$ is better. Besides, the latter is been found on the trace of the bound $2k+1$ in \cite[Lemma 2.4]{BCD}, that, in this case, coincides with (\ref{boundLP}) for $h=2$. It is therefore natural suppose that it is possible to find a bound of the form $\alpha\cdot k+1$, with $\alpha$ depending on $h$, for the cardinality of a set $\Gamma$ that is $CB(k)$, which forces $\Gamma$ to lie on on a curve of degree $h-1$. In addition, this bound should improve (\ref{boundLP}) for a range of low values of $k$ which should enlarge as $h$ growing. It is moreover an interesting question understanding, in case of affirmative answer, if there exists a function of $h$ that describe $\alpha(h)$ (see also \cite[Questions 7.3 and 7.4]{LU}).
\end{remark}

\begin{remark}
We expect that the answer to Question \ref{question} would be affirmative in any $\mathbb{P}^n$ even for values of $h$ higher than $5$. The techniques used in this paper involve, on the one hand, the study of cases in which a curve of degree $h-1$ may reduce and, on the other hand, they require the study of a suitable curve of higher degree passing though all the points of $\Gamma$. For both of these situations the number of cases to be analysed grows very quickly with the increase of $h$. For this reason, we believe that these techniques, although they could work, are not the most appropriate to deal with this issue.
\end{remark}
\bigskip

\section{Applications}

Aim of this section is to prove Theorems \ref{THlinearseries}, \ref{THcorrispnull} and to state and prove the results about complete intersection varieties we mentioned in the introduction.

\subsection{Linear series on curves}

In order to prove Theorem \ref{THlinearseries}, we need the following lemma about the Cayley-Bacharach property for a general divisor on a suitable linear series on a curve that moves on a smooth surface in $\mathbb{P}^3$.

\begin{lemma}\label{grnSP}
Let $S\subset\mathbb{P}^3$ be a smooth surface, $C$ an integral curve on $S$ and $g^r_n$ a base point free special linear series on $C$ that is not composed with an involution if $r\geq 2$. Then the general divisor $D\in g^r_n$ satisfies the Cayley-Bacharach property with respect to the dualizing sheaf of $C$. Moreover, if $|\mathcal{O}_C\otimes \mathcal{O}_S(C)|$ is base point free then $D$ also satisfies the Cayley-Bacharach property with respect to $|K_S|$.
\end{lemma}

\begin{proof}
See \cite[Lemma 3.1]{LP}.
\end{proof}

Combining the previous lemma with Theorem \ref{Th5} we can extend \cite[Theorem 1.5]{LP} proving Theorem \ref{THlinearseries}.

\begin{proof}[Proof of Theorem \ref{THlinearseries}] By Lemma \ref{grnSP}, the general divisor $D\in g^r_n$ is $CB(d-4)$. By \cite[Lemma 2.5]{LP} and  Theorem \ref{Th5} there exists an integer $h$, with $1\leq h\leq 4$, such that $h(d-h-1)\leq n \leq(h+1)(d-h-2)-1$ and $D$ lies on a curve $E$ of degree $h$. Since $S$ is smooth and of general type, it  does not contain infinitely many curves of degree $h\leq 3$. Thus,  by Bezout's theorem, we also have $n\leq hd$ when $h\leq 3$. Indeed, if $n\geq hd+1$, then a component of $E$ would be contained in $S$ and, since the $g^r_n$ is base point free, $S$ would be covered by the family of such components. Moreover, the same argument holds in the case $h=4$ when the curves of the family covering the surface $S$ are not degenerate. Finally, if the degree is $h=4$ and the curves of the family are degenerate, then the gonality of these curves is at most $3$; i.e., using the terminology of \cite{BDELU}, cov.gon$(X)\leq 3$. By \cite[Theorem A]{BDELU} it follows that cov.gon$(X)\geq d-2$, and hence $d\leq 5$. But since we are in the case $h=4$, (\ref{hgrnOnSurface}) implies that $d\geq 11$, a contradiction.
\end{proof}

\subsection{Correspondences with null trace}\label{subcor}

Theorem \ref{THcorrispnull} concerns correspondences with null trace. Before presenting the proof of the theorem, we briefly describe this notion referring the reader to \cite{LP} and \cite{B} for further details.

Let $X,Y$ be two projective varieties of dimension $n$, with $X$ smooth and $Y$ integral. A \emph{correspondence of degree $d$} on $Y\times X$ is an integral $n$-dimensional variety $\Sigma\subset Y\times X$ such that the projections $\pi_1:\Sigma\to Y, \pi_2:\Sigma\to X$ are generically finite dominant morphisms and $\deg\pi_1=d$. Let $U\subset Y_{\text{reg}}$ be an open subset such that $\dim\pi_1^{-1}(y)=0$ for every $y\in U$. Associate to $\Sigma$ there is a map $\gamma:U\to X^{(d)}$, where $X^{(d)}$ is the $d$-fold symmetric product\footnote{Let $S_d$ be the symmetric group of $d$ elements. The $d$-fold symmetric product of a variety $X$ is $X^{(d)}=X^d/S_d$.}, defined by $\gamma(y)=P_1+\dots+P_d$, where $\pi_1^{-1}(y)=\{(y,P_i)| i=1,\dots,d\}$. In  \cite[Section 2]{M} Mumford defines a \emph{trace map} $Tr_{\gamma}:H^{n,0}(S)\to H^{n,0}(U)$ (see also \cite[Section 2]{LP} and \cite[Section 4]{B}) linked to the map $\gamma$. We say that $\Sigma$ is a \emph{correspondence with null trace} if $Tr_{\gamma}=0$.\\

We are now ready to prove the following proposition that is a slightly stronger version of Theorem \ref{THcorrispnull}.

\begin{proposition}\label{strongTHcorrispnull}
Let $n\geq 3$ and let $X\subset\mathbb{P}^n$ be a smooth hypersurface of degree $d\geq n+2$. Let $\Gamma$ be a correspondence of degree $m$ with null trace on $Y\times X$. If $m\leq h(d-n-h+2)-1$ for $2\leq h\leq 5$, then the only possible values of $m$ are $(s-1)(d-n-s+3)\leq m\leq (s-1)d$ for $2\leq s\leq h$.
\end{proposition}

\begin{proof} The crucial point is that having null trace imposes a Cayley-Bacharach condition upon a correspondence. More precisely, since $K_X=(d-n-1)H$, where $H$ is an hyperplane section, then by \cite[Proposition 4.2]{B} the set $\Gamma=\pi_2(\pi_1^{-1}(y))=\{P_1,\dots,P_m\}$ is $CB(d-n-1)$, thus we get $m\geq d-n+1$ by Lemma \ref{minumCB}. Let us suppose now, by contradiction, that
\begin{equation}\label{rangeh}
(h-1)d+1\leq m \leq h(d-h-n+2)-1
\end{equation}
 for $2\leq h\leq 5$. By Theorem \ref{ThA}, $\Gamma$ lies on a curve of degree $h-1$. By Bezout's theorem this curve must have a component $E$ of degree $e\leq 4$ contained in $X$. As the  fiber $\pi_1^{-1}(y)$ is general, $X$ is covered by a family of curves of degree $e$. If $e\leq 3$,  the curves of the family would be either rational or elliptic and this is non possible since $X$ is smooth and of general type. The same holds if $e=4$ and the curves of the family are non-degenerate. On the other hand, if $e=4$ and the curves of the family are degenerate, these have gonality at most $3$, i.e., using as above the terminology of \cite{BDELU}, cov.gon$(X)\leq 3$. By \cite[Theorem A]{BDELU}, it follows that cov.gon$(X)\geq d-n+1$, and hence $d\leq n+2$. But since we are in the case $h=5$, (\ref{rangeh}) implies that $d\geq 5n+17$, a contradiction. \end{proof}

\subsection{Plane configurations}\label{subPlaneConf} Followig \cite{LU}, we call a union $\mathcal{P}=P_1\cup\dots\cup P_r\subset \mathbb{P}^n$ of positive-dimensional linear spaces a \emph{plane configuration of dimension} $\dim(\mathcal{P})=\sum\dim (P_i)$ \emph{and length} $\ell(\mathcal{P})=r$.

\begin{remark}\label{improveConjLU}
In \cite[Conjecture 1.2]{LU}, Levinson and Ullery  conjectured that given a set $\Gamma\subset\mathbb{P}^n$ of distinct points $CB(k)$, if $|\Gamma|\leq (t+1)k+1$, then $\Gamma$ lies on a plane configuration $\mathcal{P}=P_1\cup\dots\cup P_r$ of dimension $t$. In the same paper they proved the conjecture for some lower values of $t$ and $k$, given at the same time an upper bound for the length of $\mathcal{P}$. In particular, they proved the conjecture for any $k\leq 2$ and any $t\leq 3$ (cf. \cite[Theorem 1.3 (i) and (ii)]{LU}) and in the case $t=4$ and $k=3$ (cf. \cite[Theorem 1.3 (iii)]{LU}). Theorems \ref{Th5} extends \cite[Theorem 1.3]{LU}. Indeed, for $k\geq 13$, it ensures that if $|\Gamma|\leq 5k-11$, then $\Gamma$ lies on a plane configuration of dimension $4$ (with $k\leq 12$ we have $5k-11 \leq 4k+1$ and so $\Gamma$ would lie, a fortiori, on a plane configuration of dimension $3$ by \cite[Theorem 1.3(ii)]{LU}). However, we would like to point out that Theorem \ref{Th5} does not prove \cite[Conjecture 1.2]{LU} in the cases $t=4$ and $k\geq 13$. In fact, with $t=4$ we have $5k-11 < 5k+1$, that is, the upper bound for $|\Gamma|$ in our theorem is less than the one that appears in the conjecture (even if, by contrast, we get a stronger thesis).
\end{remark}

Applying \cite[Conjecture 1.2]{LU} in the case $t=3$ (which has been proved in \cite[Theorem 1.3 (ii)]{LU}), in \cite[Theorem 1.4]{LU} the authors study the geometry of fibers of some maps from complete intersections varieties. Namely, they prove that if $X\subset\mathbb{P}^{n+2}$ is a complete intersection of a quartic and a hypersurface of degree $a\geq 4n-5$ and $f:X\dashrightarrow \mathbb{P}^n$ is a finite rational map of degree at most $3a$, then the general fiber of $f$ lies on a plane configuration of dimension $3$. By leveraging the extension of \cite[Theorem 1.3]{LU} stated in Remark \ref{improveConjLU}, we can give some conditions for the $4$-dimensional plane configuration case. In fact, we can prove the following more general result.

\begin{proposition}\label{planeconfOfCorresp}
Let $X\subset\mathbb{P}^{n+2}$ be a smooth complete intersection of hypersurfaces $Y_a$ and $Y_b$, of degrees $a$ and $b$ respectively. Then for any correspondence on $X$ of degree $m$ with null trace such that $m\leq 5(a+b-n)-26$, the general fiber lies on a plane configuration of dimension $4$.
\end{proposition}

\begin{proof}
Let $\Gamma$ be a general fiber of the correspondence. Then, by \cite[Proposition 4.2]{B}, $\Gamma$ satisfies the Cayley-Bacharach condition with respect to $K_X$; that is, $\Gamma$ is $CB(a+b-n-3)$ since $X$ is the complete intersection of $Y_a$ and $Y_b$. Hence the proposition follows by Remark \ref{improveConjLU}.
\end{proof}

In particular, since a dominant rational map $X\dashrightarrow \mathbb{P}^n$ of degree $m$, from an $n$-dimensional variety $X$, gives rise to a correspondence of degree $m$ with null trace (see \cite[Example 4.6]{B}\footnote{Here the example is provided for $C^{(k)}$, but the argument works for the wider $n$-dimensional variety.}), we get the following corollary that partially extends \cite[Theorem 1.4(a)]{LU}.

\begin{corollary}\label{corrispAndPlanConf}
Let $n\geq 6$ and let $X\subset\mathbb{P}^{n+2}$ be the complete intersection of a quartic hypersurface and a  hypersurface of degree $\big\lceil \frac{5}{2}n+3 \big\rceil \leq a \leq 4n-6$. If $f:X\dashrightarrow \mathbb{P}^n$ is a finite rational map of degree at most $3a$, then the general fiber of $f$ lies on a plane configuration of dimension $4$.
\end{corollary}

\begin{proof}
Let $\Gamma$ be a general fiber of $f$. Then, by \cite[Proposition 4.2]{B}, $\Gamma$ satisfies the Cayley-Bacharach condition with respect to $K_X$; that is, $\Gamma$ is $CB(a-n+1)$ since $X$ is a complete intersection. The assumption $a\geq \big\lceil \frac{5}{2}n+3 \big\rceil$ ensures that $3a\leq 5(a-n+1)-11$. Hence the corollary follows by Proposition \ref{planeconfOfCorresp}.
\end{proof}

\begin{remark}
The previous corollary even holds without the assumption $a \leq 4n-6$. Nevertheless, if $a\geq 4n-5$, then \cite[Theorem 1.4 (a)]{LU} ensures that the fiber $\Gamma$ lies on a plane configuration of dimension $3$. Moreover, condition $n\geq 6$ just needs to guarantee that  $\big\lceil \frac{5}{2}n+3 \big\rceil \leq 4n-6$.
\end{remark}

\begin{remark}
The assumptions $a\geq \big\lceil \frac{5}{2}n+3 \big\rceil$ and $n\geq 6$ in Corollary \ref{corrispAndPlanConf} ensure that $a-n+1\geq 13$. Hence \cite[Theorem 1.4(a)]{LU} does not imply that the fiber of $f$ lies on a plane configuration of dimension $3$ (cf. Remark \ref{improveConjLU}). Corollary \ref{corrispAndPlanConf} is therefore a real extension of \cite[Theorem 1.4(a)]{LU}.
\end{remark}

\begin{remark}
Clearly, Corollary \ref{corrispAndPlanConf} holds for any $X\subset\mathbb{P}^{n+2}$ complete intersection of hypersurfaces $Y_a$ and $Y_b$ of degrees $a$ and $b$ respectively and for any finite rational map $f:X\dashrightarrow \mathbb{P}^n$ of degree $m\leq 5(a+b-n)-26$. However, the formulation of Corollary \ref{corrispAndPlanConf} allows to deduce an analogues of \cite[Theorem 1.4(b)]{LU}. Namely, in this setting the conclusion of the corollary holds for any dominant rational map $f:X\dashrightarrow \mathbb{P}^n$ of minimum degree. In fact, since $n\geq 6$ the quartic hypersurface contains a line $\ell$ (see e.g. \cite{CK}). Now, if the hypersurface $Y$ of degree $a$ does not contain the line $\ell$, then $X\cap \ell$ has length $a$, so projection from $\ell$ yields a dominant rational map $X\dashrightarrow \mathbb{P}^n$ of degree $3a$. If, instead, $\ell\subset Y$, we can still consider the projection from $\ell$, but in this case it has degree $3(a-1)$; in any case less that $3a$.
\end{remark}

\begin{remark}
Analogously, following \cite[Example 4.7]{B}, it is possible to define a correspondence of degree $d$ with null trace from a family of $d$-gonal irreducible curves covering an $n$-dimensional variety $X$. This leads us to apply Proposition \ref{planeconfOfCorresp} in the following way. Let $X\subset\mathbb{P}^{n+2}$ be the complete intersection of hypersurfaces $Y_a$ and $Y_b$, of degrees $a$ and $b$ respectively. Let $T$ be a $(n-1)$-dimensional smooth variety and let $\mathcal{E}=\{E_t\}_{t\in T}$ be a family of irreducible $d$-gonal curves covering $X$, i.e. for any general point $x\in X$ there exists $t\in T$ such that $x\in E_t$ and for any $t\in T$ there exists an holomorphic map $f_t:E_t\to \mathbb{P}^1$ of degree $d$. If $d\leq 5(a+b-n)-26$, then the fiber $f_t^{-1}(z)=\{P_1,\dots,P_d\}$ lies on a plane configuration of dimension $4$ for any $t\in T$ and for any $z\in\mathbb{P}^1$.
\end{remark}
\medskip

\section*{Acknowledgements}
 I am grateful to my Ph.D. supervisor Francesco Bastianelli for getting me interested in these problems and for his patient support. I would also like to thank Andreas Leopold Knutsen for helpful suggestions.

\bigskip

\textsc{Dipartimento di Matematica, Universit\'a degli Studi di Bari ``Aldo Moro", Via Edoardo Orabona 4, 70125 Bari -- Italy}

\emph{E-mail address}: nicola.picoco@uniba.it


\medskip


\begin{thebibliography}{99}

\bibitem[B]{B} F. Bastianelli, On the symmetric products of curves, \emph{Trans. Amer. Math. Soc.}, \textbf{364}(5) (2012), 2493--2519.

\bibitem[BDELU]{BDELU} F. Bastianelli, P. De Poi, L. Eil, R. Lazarsfeld and B. Ullery,  Measure of irrationality for hypersurfaces of large degree, \emph{Compos. Math.}, \textbf{153}(11) (2017), 2358--2393.

\bibitem[BCD]{BCD}F. Bastianelli, R. Cortini and P. De Poi, The gonality theorem of Noether for hypersurfaces,   \emph{J. Algebraic Geom.}, \textbf{23}(2) (2014), 313--339.

\bibitem[CK]{CK} C. Ciliberto and M. Zaidenberg, On Fano schemes of complete intersections, \emph{arXiv: 1903.11294v2} (2019).

\bibitem[EGH]{EGH} D. Eisenbud, M. Green and J. Harris, Cayley-Bacharach theorems and conjectures, \emph{ Bull. Amer. Math. Soc. (N.S.)}, \textbf{33}(3) (1996), 295--324.

\bibitem [G]{G} S. Greco, Remarks on the postulation of zero-dimensional subschemes of projective space,  \emph{Math. Ann.}, \textbf{284} (1989), 343--351.

\bibitem[GK]{GK} F. Gounelas and A. Kouvidakis,  Measures of irrationality of the Fano surface of a cubic threefold, \emph{Trans. Amer. Math. Soc.}, \textbf{371}(10) (2019), 7111--7133.

\bibitem[GLP]{GLP} L. Gruson, R. Lazarsfeld and C. Peskine, On a theorem of Castelnuovo and the equations defining space curves, \emph{Invent. Math.}, \textbf{72} (1983), 491--506.

\bibitem[H1]{H1} R. Hartshorne,  \emph{Residues and duality}, Lecture Notes in Math., vol. 20, Springer, 1966.

\bibitem[H2]{H2} R. Hartshorne, \emph{Algebraic geometry}, Graduate Texts in Mathematics, vol. 52, Springer, New York, NY-Heidelberg-Berlin, 1983.

\bibitem[H86]{H86} R. Hartshorne,  Generalized divisors on Gorenstein curves and a theorem of Noether, \emph{J. Math. Kyoto Univ.},  \textbf{26}(3) (1986), 375--386.

\bibitem[Ha]{Ha} J. Harris, \emph{Curves in projective space}, Les Presses de l'Universit\'e de Montreal, 1982.

\bibitem[KLR]{KLP} M.Kreuzer, L. N. Long and L. Robbiano, On the Cayley-Bacharach property, \emph{Comm. Algebra}, \textbf{47}(1) (2019), 328--354.

\bibitem[LM]{LM} R.  Lazarsfeld and O. Martin, Measures of association between algebraic varieties, \emph{arXiv:2112.00785v1} (2021).

\bibitem[LP]{LP} A. F. Lopez and G. P. Pirola, On the curves through a general point of smooth surface in  $\mathbb{P}^3$, \emph{Math. Z.}, \textbf{219} (1994), 93--106.

\bibitem[LU]{LU} J. Levinson and B. Ullery, A Cayley-Bacharach theorem and plane configurations, \emph{Proc. Amer. Math. Soc.} to appear.

\bibitem[M]{M} D. Mumford, Rational equivalence of $0$-cycles on surfaces,  \emph{J. Math. Kyoto Univ}, \textbf{9}(2) (1969),195--204.

\bibitem[SU]{SU} D. Stapleton and B. Ullery, The degree of irrationality of hypersurfaces in various Fano varieties, \emph{Manuscripta Math.}, \textbf{161}(3-4) (2020), 377--408.

\end{thebibliography}
\end{document}